\documentclass[11pt]{amsart}
\usepackage{amsfonts}
\usepackage{bbm}
\usepackage{amsfonts,amssymb,amsmath,amsthm}
\usepackage{url}
\usepackage{enumerate}
\usepackage{bbm}
\usepackage[all]{xy}
\usepackage[pdftex, colorlinks, citecolor=red, backref=page]{hyperref}
\usepackage{color,xcolor}
\usepackage{dsfont}
\usepackage{lmodern}
\usepackage{microtype}
\usepackage{mathtools}
\usepackage{enumitem}

\newtheorem{theorem}{Theorem}[section]
\newtheorem{lemma}[theorem]{Lemma}
\newtheorem{definition}[theorem]{Definition}

\newtheorem{proposition}[theorem]{Proposition}
\newtheorem{corollary}[theorem]{Corollary}
\theoremstyle{definition}
\newtheorem{remark}[theorem]{Remark}

\newtheorem{example}[theorem]{Example}
\newcommand{\dist}{\operatorname{dist}}

\newcommand{\norm}[1]{\lVert #1\rVert}

\author{Qingnan An}
\address{School of Mathematics and Statistics, Northeast Normal University, Changchun, {\rm130024}, China}
\email{qingnanan1024@outlook.com}
\author{Zhichao Liu}
\address{School of Mathematical Sciences,
Dalian University of Technology,
Dalian, {\rm116024}, China }
\email{lzc.12@outlook.com}

\keywords{Complexity rank one; Uniform Roe algebra; Real rank zero; }

\subjclass[2020]{Primary 46L85, Secondary 46L05, 46L80, 58B34}
\begin{document}

\title[Complexity rank one implies real rank zero] {Complexity rank one implies real rank zero}

\begin{abstract}
We show that every unital $C^*$-algebras with complexity rank one has real rank zero. As an application, we prove that if $X$ is a bounded geometry metric space with asymptotic dimension at most one, then the uniform Roe algebra  $C_u^*(X)$ has real rank zero. In particular, the algebras $C_u^*|\mathbb{Z}|$ and $C_u^*|\mathbb{F}_n|$ have real rank zero. These results settle a series of open questions. %This provides  affirmative answers to some open questions.   %this answers \cite[Question 6.5]{JW} and \cite[Question 3.10]{LW}. %As a application, there is a stable finite, real rank zero $C^*$-algebra but doesn't have stable rank one, which answers a long standing open question.
\end{abstract}

\maketitle
\section*{Introduction}

The geometric concept of finite decomposition complexity was originally introduced by Guentner, Tessera, and Yu \cite{GTY} to study the topological rigidity of manifolds, such as the stable Borel conjecture. This concept is inspired by, and serves as a far-reaching generalization of, Gromov's finite asymptotic dimension \cite{G1}—a notion that has proven vital for K-theory computations in both topological and analytic settings. Yu proved the strong Novikov Conjecture for finitely generated groups with finite asymptotic dimension \cite{Yu98} and for groups coarsely embeddable into Hilbert spaces \cite{Yu00}, which led to many substantial works \cite{CG,GY,GTY,STY,Yu11}. Bridging coarse geometry and topological dynamics, this decomposition method was adapted into finite dynamical complexity \cite{GWY}, yielding tools like controlled K-theory to inductively decompose and compute crossed product $C^*$-algebras. Motivated by the success of decomposability in both coarse geometry \cite{GTY,GTY2} and dynamics \cite{GWY}, Willett and Yu \cite{WY} introduced the concept of decomposability for $C^*$-algebras over a specified class of $C^*$-algebras.
A $C^*$-algebra is defined to have complexity rank zero if it is locally finite-dimensional, and it has complexity rank
at most $n+1$ if it decomposes over the class of $C^*$-algebras of complexity rank at most $n$. Notably, numerous operator algebras arising naturally in coarse geometry, groupoids, and dynamical systems have finite complexity.

The primary application of finite complexity rank relates to the well-known  Universal Coefficient Theorem (UCT) \cite{RS}.
Specifically, Willett and Yu \cite{WY} proved that all separable nuclear $C^*$-algebras satisfy the UCT if and only if any unital Kirchberg algebra with zero K-theory decomposes over the class of finite-dimensional $C^*$-algebras. Since a unital Kirchberg algebra is not locally finite-dimensional, the latter condition is equivalent to having complexity rank one. There are abundant connections between complexity rank, real rank and finite nuclear dimension \cite{JW}. Prior to this work, it was unknown whether complexity rank one implies real rank zero. This question is also connected to several broader open problems; see \cite{JW,LW,WY}.

A closely related area is the study of uniform Roe algebras, where $C^*_u(X)$ reflects the coarse geometry of a discrete metric space $X$ with bounded geometry. While $X$ having asymptotic dimension zero corresponds to $C^*_u(X)$ being locally finite-dimensional, real rank zero need not hold when the  asymptotic dimension is two or higher, such as $C^*_u|\mathbb{Z}^2|$. This leaves the one-dimensional case as a special open gap. In particular, it was unknown whether $C_u^*\vert\mathbb{Z}\vert$ or $C_u^*\vert \mathbb{F}_2\vert$ has real rank zero; see Question 3.10 and the preceding paragraph in \cite{LW}.

In this paper, we answer the complexity rank part of \cite[Question 6.5]{JW}. The proof combines the finite-dimensional overlap formulation of complexity rank at most one with Lin’s theorem on almost commuting self-adjoint matrices, a Sylvester estimate, and a three-corner spectral decomposition. We prove that every unital $C^*$-algebra with complexity rank at most one has real rank zero. %Combining this result with the %coarse-groupoid estimate relating the complexity rank of %$C_u^*(X)$ to the geometric complexity rank of $X$,
Furthermore, we show that $C_u^*(X)$ has real rank zero whenever $X$ is a bounded geometry metric space with  asymptotic dimension at most one. This implies that the algebras $C_u^*|\mathbb Z|$ and  $C_u^*|\mathbb{F}_n|$  have real rank zero, thereby giving affirmative answers to the corresponding questions raised in \cite[p. 111]{LW}.

%In this paper, we resolve the open problem The proof combines the finite-dimensional overlap formulation of complexity rank one with Lin’s theorem on almost commuting self-adjoint matrices and a three-corner spectral estimate. Specifically, our central technique allows us to prove that complexity rank one implies real rank zero. Applying this algebraic result to the coarse geometric setting, we note that the complexity rank of $C^*_u(\vert{}X\vert{})$ is controlled by the geometric complexity of $X$, ensuring that $C^*_u(\vert{}\mathbb{Z}\vert{})$ has complexity rank one. As a direct corollary of our main theorem, we establish that $C^*_u(\vert{}\mathbb{Z}\vert{})$ has real rank zero, thereby affirmatively answering \cite[Question 3.10]{LW} and contributing a partial solution to \cite[Question 6.5]{JW}.

%Moreover, our results shed light on higher-dimensional spaces. It is known that $C^*_u\vert{}\mathbb{Z}^2\vert{}$ has complexity rank at most two but lacks real rank zero; hence, our theorem implies that its complexity rank must be exactly two.

Moreover, our results also pin down the complexity rank of some higher-dimensional examples. Since $C^*_u\vert{}\mathbb{Z}^2\vert{}$ is known to have complexity rank at most two and to lack real rank zero, our theorem forces its complexity rank to be exactly two.

%The paper is organized as follows. Section 1 introduces the complexity rank and preliminaries. Section 2 establishes the Sylvester estimate used to control the  off-diagonal corner block. Section 3 develops the commuting perturbation and three-corner approximation. Section 4 proves the main result and presents some applications.

The paper is organized as follows.  Section 1 introduces the complexity rank and preliminaries. Then we prove a Sylvester estimate in Section 2 to control the off-diagonal corner block. Section 3 is devoted to the commuting perturbation argument and the three-corner approximation, and Section 4 states the main theorem and its applications.

\section{Preliminaries}

\begin{definition}\rm
  Let $A$ be a $\mathrm{C}^*$-algebra. Denote by $A_{sa}$ the set of  self-adjoint elements. %and by $A_+$ the positive cone of positive elements.
  % denote by $A^1$ the closed ball of $A$.
  For any $a,b\in A$, we write $a \approx_\varepsilon b$, if $\|a-b\|<\varepsilon$. Let $S \subset A$ be a finite set, then we write $a \in_\varepsilon S$, if there exists $s \in S$ such that $a \approx_\varepsilon s$.
  We use $[a,b]$ to denote $ab-ba$.
   %Let $B,C$ be subalgebras of $A$,  we write ``$B \subseteq_\varepsilon C$", if $b \in_\varepsilon C^1$ holds for any $b\in B^1$.
\end{definition}

\begin{definition}[\cite{BP}]\rm
  %Let $A$ be a unital $\mathrm{C}^*$-algebra. %$A$ is said to have stable rank one, if the set of invertible elements of $A$ is dense.
  A unital $C^*$-algebra $A$ is said to have real rank zero, if %the set of invertible self-adjoint elements is dense in the set $A_{sa}$ of self-adjoint elements of $A$.
  the set of self-adjoint elements with finite spectrum is norm dense in $A_{sa}$.
\end{definition}

\begin{definition}\rm
Let $\mathcal{C}$ be a class of $C^*$-algebras. A $C^*$-algebra $A$ is locally in $\mathcal{C}$ if, for any finite subset $X$ of $A$ and any $\varepsilon>0$, there is a $C^*$-subalgebra $C$ of $A$ that is in $\mathcal{C}$, and such that $x \in_\varepsilon C$ for all $x \in X$.
\end{definition}

%The following definitions are Definition 1.1 and Definition 1.3 in \cite{WY} (see also \cite{JW}).
\begin{definition} {\rm(}\cite[Definition 1.1]{WY}, \cite{JW}{\rm)}\rm \label{def ck}
Let $A$ be a unital $C^*$-algebra, and let $\mathcal{C}$ be a class of unital $C^*$-algebras. Then $A$ decomposes over $\mathcal{C}$ if, for every finite subset $X$ of the unit ball of $A$ and every $\varepsilon>0$, there exist $C^*$-subalgebras $C, D$, and $E$ of $A$ that are in the class $\mathcal{C}$ and contain $1_A$, and a positive contraction $h \in E$ such that

(i) $\|[h, x]\|<\varepsilon$ for all $x \in X$;

(ii) $h x \in_\varepsilon C,\left(1_A-h\right) x \in_\varepsilon D$, and $h\left(1_A-h\right) x \in_\varepsilon E$ for all $x \in X$;

(iii) for all $e$ in the unit ball of $E, e \in_\varepsilon C$ and $e \in_\varepsilon D$.
\end{definition}

One can think of $C$ and $D$ as being approximately (unitizations of) ideals
in $A$ such that $C+D=A$, $E$ being approximately equal to (the
unitization of) $C\cap D$,  and the pair $\{h,1_A-h\}$ as a partition of unity.

\begin{definition}{\rm(}\cite[Definition 2.3]{JW}{\rm)} \rm  Let $\alpha$ be an ordinal number.

(i) If $\alpha=0$, let $\mathcal{D}_0$ be the class of unital $C^*$-algebras that are locally finite-dimensional.

(ii) If $\alpha>0$, let $\mathcal{D}_\alpha$ be the class of unital $C^*$-algebras that decompose over $C^*$-algebras in $\bigcup_{\beta<\alpha} \mathcal{D}_\beta$.

A unital $C^*$-algebra has \emph{finite complexity} if it is in $\mathcal{D}_\alpha$ for some $\alpha$, in which case its \emph{complexity rank} is the smallest possible $\alpha$. We say a unital $C^*$-algebra has complexity rank zero if and only if it is locally finite-dimensional.

\end{definition}
\begin{remark}
  The class of $C^*$-algebras $\mathcal{D}_\alpha$ is closed under taking finite direct sums, quotients and inductive limits with unital connecting maps. Any unital $C^*$-algebra that is locally in $\mathcal{D}_\alpha$
is in $\mathcal{D}_\alpha$; see  \cite[Lemmas 2.5--2.7]{JW}.
\end{remark}

%We use the finite-spectrum characterization of real rank zero.
%\begin{theorem}[\cite{BP}]\label{thm:BP}
%For a unital $C^*$-algebra $A$,  $A$ has real rank zero iff the self-adjoint elements of $A$ with finite spectrum are norm dense in $A_{sa}$.
%\end{theorem}

The following strengthened formulation is used throughout.
%It is the finite-dimensional formulation of
%\cite[Proposition~2.14]{JW}.
\begin{proposition}\label{prop:strong-rank-one} {\rm (}\cite[Proposition~2.16]{JW}{\rm )}
A unital $C^*$-algebra $A$ has complexity rank at most one if and only if it has the following property.

For any finite subset $X$ of the unit ball of $A$ and any $\varepsilon>0$, there are unital finite-dimensional $C^*$-subalgebras
$C,D,E\subseteq A$ that contain the unit $1_A$
and a positive contraction $h\in E$ such that

{\rm (i)} $\norm{[h,x]}<\varepsilon$ for all $x\in X$;

{\rm (ii)} $hx\in_\varepsilon C,\,\,(1_A-h)x\in_{\varepsilon} D,\,\, h(1_A-h)x\in_{\varepsilon}E$ for all $x \in X$;

{\rm (iii)} $E\subseteq C\cap D$.
\end{proposition}

The exact inclusions in part~\textup{(iii)} are essential.  The
original definition is expressed in terms of approximate inclusions
and locally finite-dimensional pieces.  The passage to the form above
is  a substantial reduction that allows the corner constructions in Lemmas~\ref{lem:correction} and~\ref{lem:three-corner}.

%For a full matrix algebra, this is Lin's theorem~\cite{Lin}; Friis and
%R{\o}rdam give a short proof in~\cite{FR}.  For a finite direct sum,
%the assertion follows by applying the matrix theorem to each summand
%with the same constant.  The argument below does not require the
%approximating elements $x'$ and $y'$ to be contractions.

\section{A Sylvester estimate}

The following lemma records the precise spectral
statement.

\begin{lemma}
\label{lem:spectral-exponential}
Let $B$ be a unital $C^*$-algebra, let $x=x^*\in B$, and let
$m,M\in\mathbb R$.  For every $t\geq0$, the following statements hold.

{\rm (i)} If $x\leq M1_B$, then
$
 \norm{\exp_B(tx)}
 =\exp\bigl(t\max\sigma_B(x)\bigr)
 \leq e^{tM}.
$

{\rm (ii)} If $x\geq m1_B$, then $ \norm{\exp_B(-tx)}
 =\exp\bigl(-t\min\sigma_B(x)\bigr)
 \leq e^{-tm}.$
\end{lemma}

\begin{proof}
(i)  Since $x\leq M1_B$, the element
$M1_B-x$ is positive. By the spectral mapping theorem,
$$
       \sigma_B(M1_B-x)
       =\{M-s:s\in\sigma_B(x)\}\subseteq[0,\infty).
$$
It follows that every $s\in\sigma_B(x)$ satisfies $s\leq M$.
Equivalently,
$$
       \max\sigma_B(x)\leq M.
$$
By the spectral mapping theorem, we have
$$
       \sigma_B(\exp_B(tx))
       =\{e^{ts}:s\in\sigma_B(x)\}.
$$
Thus, $\exp_B(tx)$ is a positive element, and  we obtain
$$
\begin{aligned}
 \norm{\exp_B(tx)}
 &=\max\{e^{ts}:s\in\sigma_B(x)\}\\
 &=\exp\bigl(t\max\sigma_B(x)\bigr)
 \leq e^{tM}.
\end{aligned}
$$

(ii) If $x\geq m1_B$, then $x-m1_B$ is positive.  Similarly, we have
$$
       \min\sigma_B(x)\geq m.
$$
Applying the spectral mapping theorem again, we have
$$
\begin{aligned}
 \norm{\exp_B(-tx)}
 &=\max\{e^{-ts}:s\in\sigma_B(x)\}\\
 &=\exp\bigl(-t\min\sigma_B(x)\bigr)
 \leq e^{-tm}.
\end{aligned}
$$
%which proves part~\textup{(ii)}.

%For the corner statement, $pAp$ is a unital $C^*$-algebra with unit
%$p$.  Positivity in $pAp$ agrees with the order inherited from $A$.
%One direction is immediate.  Conversely, suppose that $y\in pAp$ is
%positive in $A$.  Then $p$ commutes with $y$ and, by continuous
%functional calculus, with $y^{1/2}$.  Since $(1-p)y=0$ and the
%square-root function vanishes at zero, functional calculus gives
%$(1-p)y^{1/2}=0$.  Hence $y^{1/2}=py^{1/2}p\in pAp$, so $y$ is
%positive in $pAp$.  Thus an inequality such as $x\leq Mp$ has the
%same meaning in either algebra.  Applying parts~\textup{(i)} and
%\textup{(ii)} with $B=pAp$ proves the corner estimates.  No
%positivity assumption on $x$ is used.
\end{proof}
\begin{remark}\label{rmk:ineq}
In particular, if $p$ is a nonzero projection in a $C^*$-algebra
$A$, then the same conclusions hold in the unital corner $pAp$, with
$p$ in place of $1_B$.  Thus, for $x=x^*\in pAp$,
$$
 x\leq Mp
 \quad\Longrightarrow\quad
 \norm{\exp_{pAp}(tx)}\leq e^{tM}
$$
and
$$
 x\geq mp
 \quad\Longrightarrow\quad
 \norm{\exp_{pAp}(-tx)}\leq e^{-tm}.
$$
The order inequalities may be interpreted either in $pAp$ or in $A$. The subscripts specify the unital algebra in which the exponential and the spectrum are computed.

\end{remark}
%\begin{remark}\label{rem:exponential-order}
%Lemma~\ref{lem:spectral-exponential} does not use an order-preserving
%property of the exponential map on arbitrary self-adjoint elements;
%such a property is false in general.  The order hypothesis is used
%only to place the spectrum of one self-adjoint element below or above
%a scalar bound.  The norm estimate then follows from scalar
%monotonicity on that spectrum.
%\end{remark}

%We next estimate a corner joining two spectrally separated regions.
%For clarity, we first state some basic facts (\ref{def: bochner}--\ref{thm: Hille}) on the Bochner integral which is the natural generalization of the familiar Lebesgue integral to the vector-valued setting %and specify the meaning of the resulting improper integral
%(see \cite[Chapter~II]{DU} and \cite[Chapter 1]{HNVW}).

We shall need a few basic facts about the Bochner integral---the vector-valued analogue of the Lebesgue integral (see %\cite[Chapter~II]{DU} and 
\cite[Chapter 1]{HNVW})---which we record here as \ref{def: bochner}--\ref{prop:com}.

\begin{definition}\rm \label{def: bochner}
Let $\mathcal{E}$ be a Banach space and let $(\Omega, \Sigma, \mu)$ be a  measure space. A function $f:\Omega\to \mathcal{E}$ is called  \emph{$\mu$-simple} if there exist $x_1,x_2,\cdots,x_n\in \mathcal{E}$  and $\Sigma_1,\Sigma_2,\cdots,\Sigma_n\in \Sigma$ satisfying $\mu(\Sigma_i)<\infty$, $i=1,2,\cdots,n$, such that $f=\sum_{i=1}^{n}x_i\chi_{_{\Sigma_i}}$, where $\chi_{_{\Sigma_i}}(\omega)=1$ if $\omega\in \Sigma_i$ and $\chi_{_{\Sigma_i}}(\omega)=0$ if $\omega\notin \Sigma_i$.
%For a $\mu$-simple function 
%$f=\sum_{i=1}^{n}x_i\chi_{_{\Sigma_i}}$, 
Then we define
$$
\int_\Omega f \mathrm{~d} \mu:=\sum_{i=1}^n \mu\left(\Sigma_i\right) x_i.
$$

We say that a property holds \emph{$\mu$-almost everywhere} if there exists a $\mu$-null set $N\subset \Omega$
 such that the property holds on the complement of $N$.
 A function $f:\Omega\to \mathcal{E}$ is called \emph{$\mu$-measurable {\rm (}strongly measurable{\rm )}} if there exists a sequence of $\mu$-simple functions $(f_n)$ with $\lim_{n}\|f_n-f\|=0$ $\mu$-almost everywhere.
\end{definition}
\begin{definition}\rm
  A function $f:\Omega\to \mathcal{E}$ is $\mu$-Bochner integrable if there exists a sequence of $\mu$-simple functions $f_n:\Omega\to \mathcal{E}$ such that
%  {\rm (i)} $\lim_nf_n=f$ $\mu$-almost everywhere;
%
%  {\rm (ii)} 
$$\lim_{n\to \infty}{\int}_{\Omega}\|f_n-f\|d\mu=0.$$
Define the Bochner integral of $f$ with respect to $\mu$ by
$$
\int_\Omega f \mathrm{~d} \mu:=\lim _{n \rightarrow \infty} \int_\Omega f_n \mathrm{~d} \mu.
$$
\end{definition}
\begin{proposition}\label{prop:com}
  A strongly measurable function $f:\Omega\to \mathcal{E}$ is $\mu$-Bochner integrable if and only if
  $$
  \int_{\Omega}\|f\|d\mu<+\infty,
  $$
  and in this case we have
  $$
 \norm{\int_\Omega f\,d\mu}
 \leq \int_\Omega\norm{f}\,d\mu.
$$
\end{proposition}
Let $\mathcal{E}^*$ be the Banach space dual of $\mathcal{E}$, which is a Banach space equipped with the norm
$$
\|\varphi\|_{\mathcal{E}^*}=\sup_{\|x\|\leq 1}|\varphi(x)|.
$$
%The following is a special case of Hille.
It is immediate from the definition of the Bochner integral that
%\begin{theorem}\label{thm: Hille} Let $\varphi\in \mathcal{E}^*$.  %be a closed linear operator defined inside $\mathcal{E}$ and %having values in a Banach space $\mathcal{F}$. If $f$ and $Tf$ %are Bochner integrable with respect to $\mu$, then
%$$
%T\left( \int_{\Omega}fd\mu\right)=\int_{\Omega}Tfd\mu.
%$$
if $f$ is Bochner integrable respect to $\mu$, then $\varphi\circ f$ is Bochner integrable, and 
$$
\varphi\left( \int_{\Omega}fd\mu\right)=\int_{\Omega}\varphi\circ fd\mu.
$$
%\begin{corollary} \label{cor: sep}
%  Let $f$ and $g$ be strong measurable. If for each $\varphi\in \mathcal{E}^*$, $\varphi(f)=\varphi(g)$ $\mu$-almost everywhere, then $f=g$ $\mu$-almost everywhere.
%\end{corollary}
We shall also need the following lemma which is a direct consequence of the above facts.
\begin{lemma} \label{lem:bochner-facts}
Let $B$ be a Banach space.

{\rm (i)} If $g\colon[0,\infty)\to B$ is norm-continuous and
$$
             \int_0^\infty\norm{g(t)}\,dt<\infty,
$$
then $g$ is Bochner integrable and
$$
 \norm{\int_0^\infty g(t)\,dt}
 \leq \int_0^\infty\norm{g(t)}\,dt.
$$
%Equivalently, the truncated integrals converge in norm as their upper
%endpoints tend to infinity.

{\rm (ii)}  If $G\colon[0,r]\to B$ is continuously differentiable in norm,
then
$$
             G(r)-G(0)=\int_0^r G'(t)\,dt.
$$
%where the integral is the Bochner integral.
\end{lemma}

\begin{proof}
%Norm-continuity implies strong measurability.  The Bochner
%integrability criterion therefore shows that the integrability of
%$t\mapsto\norm{g(t)}$ implies the Bochner integrability of $g$; see
%\cite[Proposition~1.16]{vN} and \cite[Chapter~II]{DU}.
\iffalse

We include a proof of the norm inequality.  Let
$(\Omega,\Sigma,\nu)$ be a measure space.  If
$$
       s=\sum_{j=1}^n \mathbf 1_{E_j}x_j
$$
is an integrable $B$-valued simple function, we may omit zero
coefficients and assume that the sets $E_j$ are pairwise disjoint.
Integrability then implies $\nu(E_j)<\infty$ for every $j$.  Hence
$$
\begin{aligned}
 \norm{\int s\,d\nu}
 &=\norm{\sum_{j=1}^n\nu(E_j)x_j}\\
 &\leq\sum_{j=1}^n\nu(E_j)\norm{x_j}
 =\int\norm{s}\,d\nu.
\end{aligned}
$$
If $f$ is Bochner integrable, there are integrable simple functions
$s_n$ such that
$$
             \int\norm{s_n-f}\,d\nu\longrightarrow 0.
$$
By the definition of the Bochner integral,
$\int s_n\,d\nu\to\int f\,d\nu$ in norm.  Moreover,
$$
 \left|\int\norm{s_n}\,d\nu-\int\norm{f}\,d\nu\right|
 \leq\int\left|\norm{s_n}-\norm{f}\right|\,d\nu
 \leq\int\norm{s_n-f}\,d\nu\longrightarrow0.
$$
Passing to the limit in the simple-function inequality proves the first inequality.

For $0\leq a<b$, the same inequality gives
$$
 \norm{\int_0^bg(t)\,dt-\int_0^ag(t)\,dt}
 \leq\int_a^b\norm{g(t)}\,dt.
$$
Since the scalar tail on the right tends to zero, the truncated
integrals form a norm-Cauchy net.  This proves the stated
improper-integral interpretation.
\fi
Under the assumptions of (i), standard calculus shows that norm-continuity implies strong measurability.  The Bochner integrability criterion therefore follows from Proposition \ref{prop:com}.

Now we  prove \textup{(ii)}. For $\varphi\in B^*$, %Theorem \ref{thm: Hille}  implies
%$$
% \varphi\!\left(\int_0^r u(t)\,dt\right)
% =\int_0^r\varphi(u(t))\,dt.
%$$
%The identity follows first for simple functions and then for general
%Bochner-integrable functions by approximation.
applying the fundamental theorem of calculus to $\varphi\circ G$, we have
\begin{align*}
  \varphi\circ G(r)-\varphi \circ G(0) & =\int_0^r (\varphi\circ G)'(t)\,dt \\
   & = \int_0^r \varphi\circ G'(t)\,dt.
\end{align*}
Now we have
$$
 \varphi\!\left(G(r)-G(0)-\int_0^r G'(t)\,dt\right)=0,
 \quad \forall\,\varphi\in B^*.
$$
Then the Hahn--Banach theorem implies (ii).
%that $B^*$ separates the points of $B$, so (ii) follows.
\end{proof}
The following is the ordered self-adjoint case of the Sylvester--Rosenblum operator equation.  The spectral-separation method originates in \cite{Rosenblum}. % (see also \cite{BR}).  %We give the integral argument in the form required below.
\begin{lemma}\label{lem:sylvester}
Let $p,q$ be orthogonal projections in a unital $C^*$-algebra $A$ and let
$r=r^*$ commute with both $p$ and $q$.  Suppose that
$$
       prp\leq \lambda p,
       \quad
       qrq\geq \mu q
$$
for real numbers $\lambda<\mu$.  Then, for every $z\in A$,
$$
       \norm{pzq}\leq\frac{1}{\mu-\lambda}
                       \norm{p[r,z]q}.
$$
\end{lemma}

\begin{proof}
If $p=0$ or $q=0$, the inequality holds trivially.
Assume that $p$ and $q$ are nonzero, %and work in the unital corner algebras $pAp$ and $qAq$, whose units are $p$ and $q$,
%respectively.
and set
$$
       R_p=prp\in(pAp)_{sa},\quad
       R_q=qrq\in(qAq)_{sa},
$$$$
       x=pzq,\quad y=p[r,z]q.
$$
Since $r$ commutes with $p$ and $q$, one has
\begin{align*}
 R_px-xR_q
 &=prp\,pzq-pzq\,qrq\\
 &=przq-pzrq\\
 &=p(rz-zr)q
 =y.
\end{align*}

For $t\geq0$, let
$$
 U(t)=\exp_{pAp}(tR_p)\in pAp,
 \quad
 V(t)=\exp_{qAq}(-tR_q)\in qAq.
$$
%where the subscripts indicate that the exponentials are formed in the
%corner algebras.  This distinction is essential: the exponential of
%$-tR_q$ formed in $A$ contains the additional summand $1-q$ and
%therefore need not decay.
Applying Lemma~\ref{lem:spectral-exponential} in $pAp$ to the inequality
$R_p\leq\lambda p$, and in $qAq$ to $R_q\geq\mu q$, we obtain %The order
%inequalities first give the scalar spectral bounds
$$
 \max\sigma_{pAp}(R_p)\leq\lambda,
 \quad
 \min\sigma_{qAq}(R_q)\geq\mu.
$$
By  Lemma~\ref{lem:spectral-exponential} and Remark \ref{rmk:ineq},
\begin{align*}
 \norm{U(t)}
 &=\exp\!\,\bigl(t\max\sigma_{pAp}(R_p)\bigr)
 \leq e^{t\lambda},\\
 \norm{V(t)}
 &=\exp\!\,\bigl(-t\min\sigma_{qAq}(R_q)\bigr)
 \leq e^{-t\mu}.
\end{align*}
%Thus the exponentials are formed from $R_p$ and $R_q$ themselves;
%the spectral maximum and minimum appear only in the corresponding
%norm formulas.
Define $F(t)=U(t)xV(t)$.  Then
$$
       \norm{F(t)}
       \leq e^{-t(\mu-\lambda)}\norm{x},
$$
so $F(t)$ tends to zero in norm.

 Since $U(0)=p$, $V(0)=q$, and
$pxq=x$, we have $F(0)=x$.  Differentiating in the Banach
algebra $A$ yields
\begin{align*}
 F'(t)
 &=U(t)(R_px-xR_q)V(t)\\
 &=U(t)yV(t).
\end{align*}
The map $t\mapsto U(t)yV(t)$ is norm-continuous, and
$$
       \norm{U(t)yV(t)}
       \leq e^{-t(\mu-\lambda)}\norm{y},
$$
where the right side is integrable. Hence,
Lemma~\ref{lem:bochner-facts}\,\textup{(i)} shows that
$t\mapsto U(t)yV(t)$ is Bochner integrable on $[0,\infty)$.
%Equivalently, its improper integral converges in norm.
Applying
Lemma~\ref{lem:bochner-facts}\,\textup{(ii)} to $F(t)$ on $[0,r]$,  we have
$$
       F(r)-x=\int_0^r U(t)yV(t)\,dt.
$$
Taking the limit as $r\to\infty$, since $F(r)\to 0$ in norm, we obtain
$$
       -x=\int_0^\infty U(t)yV(t)\,dt.
$$
By Lemma~\ref{lem:bochner-facts}\,\textup{(i)} and the
exponential estimates, we have
$$
 \norm{x}
 \leq\int_0^\infty\norm{U(t)yV(t)}\,dt
 \leq\int_0^\infty e^{-t(\mu-\lambda)}\norm{y}\,dt
          =\frac{\norm{y}}{\mu-\lambda}.
$$
\end{proof}

\section{The three-corner estimates}

We isolate the perturbation argument that produces an exactly
commuting pair in the overlap algebra.  %The constants are selected for
%the application in Section~5 and are not intended to be optimal.
We need the following theorem of Lin.
\begin{theorem}[\cite{FR, Lin}]\label{thm:lin}
For every $\alpha>0$ there is a number
$\delta_L(\alpha)>0$ such that %for any $n$ and any self-adjoint elements $x,y\in M_n(\mathbb{C})$ satisfying $\|x\|, \|y\|\leq 1$ and  $\norm{[x,y]}<\delta_L(\alpha),$
for any finite-dimensional $C^*$-algebra $E$ and any self-adjoint elements $x,y\in E$ satisfying $\|x\|, \|y\|\leq 1$ and  $\norm{[x,y]}<\delta_L(\alpha),$
there are commuting self-adjoint elements $x',y'\in E$ such that
$$
      \norm{x-x'}<\alpha,
      \quad
      \norm{y-y'}<\alpha.$$
%The constant is independent of the dimensions and of the number of
%matrix summands of $F$.
\end{theorem}

\begin{lemma}\label{lem:correction}
Let $A$ be a unital $C^*$-algebra, let $a\in A_{sa}$ satisfy
$\norm{a}\leq1$, and let $C,D,E\subseteq A$ be unital
finite-dimensional $C^*$-subalgebras which  contain the unit $1_A$.
Suppose that $h\in E$ is a positive contraction and that

{\rm (i)} $\norm{[h,a]}<\delta$;

{\rm (ii)} $ha\in_\delta C,\,\,(1-h)a\in_{\delta} D,\,\, h(1-h)a\in_{\delta}E$;

{\rm (iii)} $E\subseteq C\cap D$.

Let $\alpha>0$, assume $\delta<\min\{\frac{1}{10}, \frac{1}{6}\delta_{L}(\alpha)\}$. Then there are a positive contraction $h_0\in E$ and an element
$e_0\in E_{sa}$ such that $[h_0,e_0]=0$, and with
$k_0=h_0(1-h_0)$,  we have

{\rm (i)$'$} $\norm{[h_0,a]}<\delta+4\alpha$;

{\rm (ii)$'$} $h_0a\in_{\delta+2\alpha} C,\,\,(1-h_0)a\in_{\delta+2\alpha} D$;

{\rm (iii)$'$} $\norm{k_0a-e_0}<\frac52\delta+7\alpha$.
\end{lemma}

\begin{proof}
Set
$$
       k=h(1-h)=h-h^2.
$$
Since $h$ is a positive contraction, $0\leq k\leq \frac14$ and
$\norm{k}\leq\frac14$. Then % Expanding the commutator gives
\begin{equation}\label{eq:k-commutator}
 [k,a]=[h,a]-h[h,a]-[h,a]h,
\end{equation}
and therefore
\begin{equation}\label{eq:k-commutator-bound}
       \norm{[k,a]}\leq3\norm{[h,a]}<3\delta.
\end{equation}

Choose $e\in E$ with $\norm{ka-e}<\delta$ and set
$$
       e_{sa}=\frac12(e+e^*)\in E_{sa}.
$$
The element $ka$ need not be self-adjoint, and $(ka)^*=ak$.  Therefore
\begin{equation}
\begin{aligned}
 \norm{ka-e_{sa}}
 &\leq \frac12\norm{ka-e}
       +\frac12\norm{ka-e^*}                                      \\
 &\leq \frac12\delta
       +\frac12\bigl(\norm{ka-ak}+\norm{ak-e^*}\bigr)             \\
 &< \frac12\delta+\frac12(3\delta+\delta)
   =\frac52\delta.
\end{aligned}
\label{eq:symmetrization}
\end{equation}
Here $\norm{ak-e^*}=\norm{ka-e}<\delta$.

Since $h$ commutes with $k$,
$$
       [h,ka]=k[h,a].
$$
Consequently, using~\eqref{eq:symmetrization},
\begin{align}
 \norm{[h,e_{sa}]}
 &\leq \norm{[h,ka]}+\norm{[h,e_{sa}-ka]} \notag\\
 &\leq \norm{k}\,\norm{[h,a]}+2\norm{e_{sa}-ka} \notag\\
 &<\frac14\delta+5\delta
   <6\delta.                         \label{eq:almost-commuting}
\end{align}
Moreover,
$$
       \norm{e_{sa}}
       \leq\norm{ka}+\norm{e_{sa}-ka}
       <\frac14+\frac52\delta\leq\frac12,
$$
where the last inequality uses $\delta<1/10$.  Hence, $h$ and
$e_{sa}$ are self-adjoint contractions in $E$.  By
Theorem~\ref{thm:lin} and
\eqref{eq:almost-commuting}, there are commuting $h',e_0\in E_{sa}$
such that
\begin{equation}\label{eq:lin-output}
       \norm{h'-h}<\alpha,
       \quad
       \norm{e_0-e_{sa}}<\alpha.
\end{equation}

Let $f\colon\mathbb R\to[0,1]$ be the truncation function
$$
       f(t)=\max\{0,\min\{t,1\}\},
$$
and set $h_0=f(h')$. Then $0\leq h_0\leq 1$. Now \eqref{eq:lin-output} implies that the spectrum of $h'$ is contained in $(-\alpha,1+\alpha)$.  Therefore
\begin{equation}\label{eq:h0-h}
       \norm{h_0-h}
       \leq\norm{h_0-h'}+\norm{h'-h}<2\alpha.
\end{equation}
Since $h'$ commutes with $e_0$, $h_0$ also commutes with $e_0$, i.e.,
$[h_0,e_0]=0$.

Let $k_0=h_0-h_0^2$.  For two positive contractions $r,s$,
$$
 \norm{r(1-r)-s(1-s)}
 \leq \norm{r-s}+\norm{r^2-s^2}
 \leq3\norm{r-s}.
$$
(The second inequality follows from
$r^2-s^2=r(r-s)+(r-s)s$; this identity does not require $r$ and $s$
to commute.)
 Hence,  \eqref{eq:h0-h} implies
\begin{equation}\label{eq:k0-k}
       \norm{k_0-k}<6\alpha.
\end{equation}

Now we check the estimates.  For the commutator,
$$
\begin{aligned}
 \norm{[h_0,a]}
 &\leq\norm{[h,a]}+\norm{[h_0-h,a]}\\
 &<\delta+2\norm{h_0-h}\norm{a}
 <\delta+4\alpha.
\end{aligned}
$$
For the two distance estimates,
$$
 \dist(h_0a,C)
 \leq\dist(ha,C)+\norm{(h_0-h)a}
 <\delta+2\alpha,
$$
and similarly,
$$
 \dist((1-h_0)a,D)<\delta+2\alpha.
$$
Finally, by~\eqref{eq:symmetrization},~\eqref{eq:lin-output}, and
\eqref{eq:k0-k},
$$
\begin{aligned}
 \norm{k_0a-e_0}
 &\leq\norm{(k_0-k)a}+\norm{ka-e_{sa}}+\norm{e_{sa}-e_0}\\
 &<6\alpha+\frac52\delta+\alpha
 =\frac52\delta+7\alpha.
\end{aligned}
$$
\end{proof}

We use the commuting pair from Lemma~\ref{lem:correction} to construct
a finite-dimensional approximation.  The proof utilizes a partition of unity via three spectral projections and estimates every off-diagonal block.

\begin{lemma}\label{lem:three-corner}
Let $A$ be a unital $C^*$-algebra and let $a\in A_{sa}$ be a
contraction.  Suppose $C,D,E\subseteq A$ are unital
finite-dimensional $C^*$-subalgebras satisfying $E\subseteq C\cap D$.
Let $h_0\in E$ be a positive contraction, put
$k_0=h_0(1-h_0)$, and suppose $e_0\in E_{sa}$ commutes with $h_0$.
If, for some $\gamma>0$,
\begin{equation}\label{eq:gamma-hypotheses}
\begin{aligned}
 \norm{[h_0,a]}&<\gamma,\\
 \dist(h_0a,C)&<\gamma,\\
 \dist((1-h_0)a,D)&<\gamma,\\
 \norm{k_0a-e_0}&<\gamma,
\end{aligned}
\end{equation}
then there is a finite-dimensional unital $C^*$-subalgebra
$F\subseteq A$ and $b\in F_{sa}$ such that
\begin{equation}\label{eq:three-corner-conclusion}
       \norm{a-b}<12\gamma.
\end{equation}
\end{lemma}

\begin{proof}
Since $E$ is finite-dimensional, $h_0$ has finite spectrum.  Define
the spectral projections
\begin{equation}\label{eq:spectral-projections}
\begin{aligned}
 p_0&=\mathds{1}_{[0,1/3)}(h_0),\\
 p_\diamond&=\mathds{1}_{[1/3,2/3]}(h_0),\\
 p_1&=\mathds{1}_{(2/3,1]}(h_0).
\end{aligned}
\end{equation}
%Although the displayed characteristic functions are discontinuous on
%$\mathbb R$, their restrictions to the finite set $\sigma(h_0)$ are
%continuous.  Each restriction extends to a continuous function on
%$\mathbb R$, so the
These three projections belong to
$C^*(h_0)\subseteq E$. % They are pairwise orthogonal and sum to $1$.
Since $e_0$ and $k_0$ commute with $h_0$, they also commute with
$p_0,p_\diamond,p_1$.

We first estimate the off-diagonal terms of $a$ relative to these
projections.  On $[1/3,2/3]$, the scalar function $t\mapsto t(1-t)$
is bounded below by $2/9$.  Hence,
$$
       p_\diamond k_0p_\diamond\geq\frac29p_\diamond.
$$
If $p_\diamond\neq0$, the element $p_\diamond k_0p_\diamond$ is invertible in the
corner $p_\diamond Ep_\diamond$.  Define
$$
r_\diamond=
\begin{cases}
 (p_\diamond k_0p_\diamond)^{-1},&p_\diamond\neq0,\\
 0,&p_\diamond=0.
\end{cases}
$$
Then
$$
r_\diamond\in p_\diamond Ep_\diamond, \quad \norm{r_\diamond}\leq\frac92,\quad
%This convention makes every formula below valid also when the middle projection vanishes.
%Because $p_\diamond$ is a spectral projection of $h_0$, it commutes with
%$k_0=h_0(1-h_0)$.
%Hence,
%$$
%       p_\diamond k_0=k_0p_\diamond=p_\diamond k_0p_\diamond,
       r_\diamond(p_\diamond k_0p_\diamond)=p_\diamond.
$$
Moreover, $e_0$ commutes with $p_\diamond$, so
$p_\diamond e_0(1-p_\diamond)=0$.  %Keep the corner projections at each step,
Now we obtain
$$
\begin{aligned}
 p_\diamond a(1-p_\diamond)
 &=r_\diamond(p_\diamond k_0p_\diamond)\,p_\diamond a(1-p_\diamond)\\
 &=r_\diamond p_\diamond k_0p_\diamond a(1-p_\diamond)\\
 &=r_\diamond p_\diamond k_0a(1-p_\diamond)\\
 &=r_\diamond p_\diamond(k_0a-e_0)(1-p_\diamond).
\end{aligned}
$$
Here, the third equality uses the identity $p_\diamond k_0p_\diamond=p_\diamond k_0$ and  the final equality
uses $p_\diamond e_0(1-p_\diamond)=0$.
It follows that
\begin{equation}\label{eq:middle-off-diagonal}
       \norm{p_\diamond a(1-p_\diamond)}<\frac92\gamma.
\end{equation}
%Since $a=a^*$, the adjoint corner has the same norm.

Now we  set
$$
       x=p_0ap_1,\qquad
       H_0=p_0h_0p_0,\qquad
       H_1=p_1h_0p_1.
$$
Then $H_0\leq\frac13p_0$, $H_1\geq\frac23p_1$, and
$$
       H_0x-xH_1=p_0[h_0,a]p_1.
$$
Then Lemma~\ref{lem:sylvester} implies
\begin{equation}\label{eq:outer-off-diagonal}
       \norm{p_0ap_1}<3\gamma.
\end{equation}
%The adjoint corner $p_1ap_0$ has the same norm because $a=a^*$.

Define the block diagonal part
$$
       a_\Delta=p_0ap_0+p_\diamond ap_\diamond+p_1ap_1.
$$
%Relative to the ordered decomposition $1=p_0+p_\diamond+p_1$,
The off-diagonal part of $a$ has the following $3$-by-$3$ block form.  The
four blocks adjacent to the middle row or column are controlled by
\eqref{eq:middle-off-diagonal}; the two outer blocks are controlled by
\eqref{eq:outer-off-diagonal}:
$$
\begin{array}{c|ccc}
 &p_0&p_\diamond&p_1\\ \hline
p_0&*&p_0ap_\diamond&p_0ap_1\\[1mm]
p_\diamond&p_\diamond ap_0&*&p_\diamond ap_1\\ [1mm]
p_1&p_1ap_0&p_1ap_\diamond&*
\end{array}
$$
%To estimate the norm of this matrix, set
%$$       X=p_\diamond a(1-p_\diamond),
%       \quad      Y=p_0ap_1.$$
%The four blocks meeting the middle region form $X+X^*$, while the two
%outer blocks form $Y+Y^*$.  Thus
Note that
\begin{equation}\label{eq:matrix-splitting}
       a-a_\Delta=p_\diamond a(1-p_\diamond)+(1-p_\diamond) ap_\diamond+x+x^*.
\end{equation}
As $p_0p_1=0$, we have $x^2=(x^*)^2=0$, and hence,
$$
       (x+x^*)^2=xx^*+x^*x.
$$
The two positive summands on the right belong to the orthogonal
corners $p_0Ap_0$ and $p_1Ap_1$. By the $C^*$-identity,
$$
 \norm{x+x^*}^2
 =\norm{xx^*+x^*x}
 =\max\{\norm{xx^*},\norm{x^*x}\}
 =\norm{x}^2,
$$
we obtain $\norm{x+x^*}=\norm{x}$.

Similarly, using
$p_\diamond (1-p_\diamond)=0$,
$$\norm{p_\diamond a(1-p_\diamond)+(1-p_\diamond) ap_\diamond}=\norm{p_\diamond a(1-p_\diamond)}.$$
 Consequently,
\eqref{eq:middle-off-diagonal},~\eqref{eq:outer-off-diagonal}, and
\eqref{eq:matrix-splitting} give the sharper estimate
\begin{equation}\label{eq:block-diagonal-error}
\begin{aligned}
 \norm{a-a_\Delta}
 &\leq{\norm{p_\diamond a(1-p_\diamond)}}+\norm{x}\\
 & <\frac92\gamma+3\gamma
 =\frac{15}{2}\gamma.
\end{aligned}
\end{equation}

Choose $c\in C$ and $d\in D$ satisfying
\begin{equation}\label{eq:c-d-choice}
       \norm{h_0a-c}<\gamma,
       \quad
       \norm{(1-h_0)a-d}<\gamma.
\end{equation}
We next approximate the three diagonal corners.

For the corner $p_0ap_0$, we use
$p_0(1-h_0)p_0\geq\frac23p_0$.  Define
$$
s_0=
\begin{cases}
 (p_0(1-h_0)p_0)^{-1},&p_0\neq0,\\
 0,&p_0=0.
\end{cases}
$$
Then $s_0\in p_0Ep_0\subset p_0Dp_0$ and $\norm{s_0}\leq\frac32$.
Set $$\widetilde b_0=s_0p_0dp_0\in p_0Dp_0.$$  Since
$$s_0p_0(1-h_0)p_0=p_0,$$ the second estimate in \eqref{eq:c-d-choice} yields
$$
 \norm{p_0ap_0-\widetilde b_0}
 =\norm{s_0p_0\bigl((1-h_0)a-d\bigr)p_0}
 <\frac32\gamma.
$$
Let $b_0=(\widetilde b_0+\widetilde b_0^*)/2$.  Since $p_0ap_0$ is
self-adjoint, we have
\begin{equation}\label{eq:low-corner}
       b_0\in(p_0Dp_0)_{sa},
       \quad
       \norm{p_0ap_0-b_0}<\frac32\gamma.
\end{equation}

For the corner $p_1ap_1$, we make use of $p_1h_0p_1\geq\frac23p_1$.  Define
$$
s_1=
\begin{cases}
 (p_1h_0p_1)^{-1},&p_1\neq0,\\
 0,&p_1=0.
\end{cases}
$$
Then $s_1\in p_1Ep_1\subset p_1Cp_1$ and $\norm{s_1}\leq\frac32$.  Set
$$\widetilde b_1=s_1p_1cp_1\in p_1Cp_1.$$
Similarly,
using the first estimate in~\eqref{eq:c-d-choice}, there exists an element $b_1\in(p_1Cp_1)_{sa}$ such that
\begin{equation}\label{eq:high-corner}
       \norm{p_1ap_1-b_1}<\frac32\gamma.
\end{equation}

For the middle corner, %set $s_m=r_m\in p_mEp_m$ and
 define
$$
       b_\diamond=r_\diamond p_\diamond e_0p_\diamond\in p_\diamond Ep_\diamond.
$$
The element $r_\diamond$ belongs to $C^*(h_0)p_\diamond$, and $e_0$ commutes with
$h_0$.  Thus $r_\diamond$ commutes with $p_\diamond e_0p_\diamond$, and $b_\diamond=b_\diamond^*$.
Moreover,
\begin{equation}\label{eq:middle-corner}
 \norm{p_\diamond ap_\diamond-b_\diamond}
 =\norm{r_\diamond p_\diamond(k_0a-e_0)p_\diamond}
 <\frac92\gamma.
\end{equation}

Since  $p_0,p_1,p_\diamond\in E\subseteq C\cap D$, the corner algebras
$p_0Dp_0$, $p_\diamond Ep_\diamond$, and $p_1Cp_1$ are finite-dimensional
$C^*$-subalgebras.  Their supports are pairwise orthogonal.  Therefore
$$
       F=p_0Dp_0\oplus p_\diamond Ep_\diamond\oplus p_1Cp_1
$$
forms a finite-dimensional $C^*$-subalgebra of $A$ with unit
$p_0+p_\diamond+p_1=1$.  The element
$
       b=b_0+b_\diamond+b_1
$
belongs to $F_{sa}$. The orthogonality of the three corners and
\eqref{eq:low-corner}--\eqref{eq:middle-corner} imply
\begin{equation}\label{eq:diagonal-approximation}
 \norm{a_\Delta-b}
 =\max\{\norm{p_0ap_0-b_0},
          \norm{p_\diamond ap_\diamond-b_\diamond},
          \norm{p_1ap_1-b_1}\}
 <\frac92\gamma.
\end{equation}
Combining~\eqref{eq:block-diagonal-error} and
\eqref{eq:diagonal-approximation}, we have
$$
       \norm{a-b}<\frac{15}{2}\gamma+\frac92\gamma
                  =12\gamma,
$$
as desired.
\end{proof}

\section{Main results}

%The following  is a basic approximation from the Stone-Weierstrass theorem, see \cite[Proposition 2.5.11]{L}.
%\begin{proposition}\label{app f}
%Let $A$ be a $C^*$-algebra and $F$ be a finite set in $C[0,1]$, then for any $\varepsilon>0$, there exists $\delta>0$ satisfying that for any $a,b\in A_+$ with $\|a\|,\|b\|\leq 1$, if $\|ab-ba\|<\delta$, then
%$$
%\|f(b)a-af(b)\|<\varepsilon,\quad f\in F.
%$$
%\end{proposition}

%The following lemma is an equivalent formulation of complexity rank one, see \cite[Proposition 2.16]{JW}.

%\begin{lemma} \rm
%A unital $C^*$-algebra $A$ has complexity rank at most one if and only if it has the following property.

%For any finite subset $X$ of the unit ball of $A$ and any $\varepsilon>0$, there exist finite-dimensional $C^*$-subalgebras $C, D$ and $E$ of $A$ that contain the unit $1_A$ and a positive contraction $h \in E$ such that

%(i) $\|[h, x]\|<\varepsilon$ for all $x \in X$;

%(ii) $h x \in_\varepsilon C,\left(1_A-h\right) x \in_\varepsilon D$ and $\left(1_A-h\right) h x \in_\varepsilon E$ for all $x \in X$;

%(iii) $E$ is contained in both $C$ and $D$.
%\end{lemma}

The following theorem the complexity rank part of \cite[Question 6.5]{JW} and %has some important applications.
yields the important applications below.

\begin{theorem}\label{main thm}
 Let $A$ be a unital $C^*$-algebra with complexity rank at most one. Then $A$ has real rank zero.
\end{theorem}

\begin{proof}
Suppose that $a\in A_{sa}$ satisfies $\norm{a}\leq1$.

Fix $\varepsilon>0$ and set
\begin{equation}\label{eq:parameter-gamma-alpha}
       \gamma=\frac{\varepsilon}{12},
       \qquad
       \alpha=\frac{\gamma}{10}.
\end{equation}
Let $\delta_{L}(\alpha)$ be the constant supplied by
Theorem~\ref{thm:lin}.  Choose $\delta>0$ such that
\begin{equation}\label{eq:parameter-delta}
       \delta<\min\left\{
          \frac1{10},\frac{\gamma}{10},
          \frac1{6}\delta_{L}(\alpha)
       \right\}.
\end{equation}

Apply Proposition~\ref{prop:strong-rank-one} to $X=\{a\}$ and
$\delta$.  We obtain unital finite-dimensional
$C^*$-subalgebras $C,D,E\subseteq A$, with $E\subseteq C\cap D$, and
a positive contraction $h\in E$ satisfying the hypothesis in
Lemma \ref{lem:correction}.  %The last condition in
%\eqref{eq:parameter-delta} is exactly~\eqref{eq:lin-smallness}.
Therefore, we have $h_0,e_0\in E_{sa}$, with
$h_0$ a positive contraction and $[h_0,e_0]=0$, satisfying
the hypotheses of
Lemma~\ref{lem:three-corner}.  Indeed,
\eqref{eq:parameter-gamma-alpha} and~\eqref{eq:parameter-delta} give
\begin{align*}
 \delta+4\alpha
 &<\frac{\gamma}{10}+\frac{4}{10}\gamma
   =\frac{1}{2}\gamma<\gamma,\\
 \delta+2\alpha
 &<\frac{\gamma}{10}+\frac{2}{10}\gamma
   =\frac{3}{10}\gamma<\gamma,\\
 \frac52\delta+7\alpha
 &<\frac52\frac{\gamma}{10}+\frac{7}{10}\gamma
   =\frac{19}{20}\gamma<\gamma.
\end{align*}
Thus all four hypotheses in~\eqref{eq:gamma-hypotheses} hold.
Lemma~\ref{lem:three-corner} produces a finite-dimensional unital
$C^*$-subalgebra $F\subseteq A$ and $b\in F_{sa}$ with
$$
       \norm{a-b}<12\gamma
                 =\varepsilon.
$$
The element $b$ has finite spectrum because it belongs to the unital
finite-dimensional algebra $F$. %Given an arbitrary $a\in A_{sa}$, apply the argument to $a/\max\{1,\|a\|\}$ and rescale the resulting finite-spectrum approximation.
Thus, %every self-adjoint element can be approximated by self-adjoint element with finite spectrum, and hence,
$A$ has real rank zero.
\end{proof}
\begin{remark}
 In \cite{JW}, the authors introduced weak complexity rank, which is bounded above by complexity rank. In general, these two ranks are not equal, see \cite[Corollary 4.2]{JW}. They compared (weak) finite complexity with  many other properties. Whether weak complexity rank one implies real rank zero remains open. %Note that the uniform Roe algebra of $\mathbb{Z}^2$ has complexity rank at most two by \cite[Ex. A.9]{WY} and does not have real rank zero by \cite[Thm. 3.1]{LW}, so it is certainly not true that (weak) finite complexity implies real rank zero in general.
\end{remark}

Many examples of spaces with finite complexity arising from groupoid theory can be found in \cite{GWY17}. Using the main result of \cite{ALSS}, the complexity rank of the transformation
groupoid associated to any free action of a virtually cyclic group on a finite-dimensional space is one; see \cite[Example A.11]{WY}. Now Theorem \ref{main thm} implies the following result.
%Combining our main result with \cite[Theorem 3.3]{BL},
\begin{corollary}
%Let $X$ be an infinite, second-countable, totally disconnected, compact Hausdorff space with finite covering dimension and the action $D_{\infty} \curvearrowright X$ of the infinite dihedral group $D_{\infty}$ on $X$ is free minimal, then $C(X) \rtimes D_{\infty}$ has real rank zero.
Let $X$ be a totally disconnected compact Hausdorff space.
Let $\Gamma \curvearrowright X$ be a free action of an infinite virtually cyclic group on $X$, then $C(X) \rtimes_r \Gamma$ has real rank zero.
\end{corollary}

\begin{definition}[\cite{LW}]\rm
A metric space $X$ has bounded geometry if for any $r>0$ there is a uniform
bound on the cardinalities of all $r$-balls in $X$. Let $X$ be a bounded geometry metric space, and let $\mathbb{C}_u[X]$ denote the *-algebra of all $X$-by-$X$ matrices $\left(a_{x y}\right)$ with uniformly bounded entries in $\mathbb{C}$ such that the propagation
$$
\operatorname{prop}(a):=\sup \left\{d(x, y) \mid a_{x y} \neq 0\right\}
$$
is finite. As $X$ has bounded geometry, $\mathbb{C}_u[X]$ acts on $\ell^2(X)$ by bounded operators. The uniform Roe algebra $C_u^*(X)$ of $X$ is the operator-norm closure of the image of $\mathbb{C}_u[X]$ in this representation.

When $X=G$ is a countable discrete group equipped with a left-invariant bounded geometry metric, we write $C_u^*|G|$ for the uniform Roe algebra of $G$ to avoid any possible confusion with the group $C^*$-algebras of $G$. %There is a well-known isomorphism $C_u^*|G| \cong \ell^{\infty}(G) \rtimes_r G$, where the action of $G$ on $\ell^{\infty}(G)$ is induced by the right translation action of $G$ on itself.
\end{definition}

%In \cite[App. A.2]{WY}, the authors showed that if $X$ is a bounded geometry metric space, then the geometric complexity of X in the sense of \cite{GTY} is an upper bound for the complexity rank of the uniform Roe algebra $C^*_u(X)$; there are other examples based on groupoid theory coming from \cite{GWY17}.The canonical isomorphism between the reduced $C^*$-algebra of the coarse groupoid and $C_u^*(X)$ is given in R. Willett and G. Yu, *Higher Index Theory*, Cambridge Studies in Advanced Mathematics **189**, Cambridge University Press, 2020, Proposition 10.29.

\begin{proposition}[\cite{WY}]\label{Prop: ranks}
Let $X$ be a bounded geometry metric space. Then the complexity rank of the uniform Roe algebra $C_u^*(X)$ is bounded above by the decomposition complexity  rank of $X$ in the sense of \cite[Definition 2.2.1]{GTY2}.
\end{proposition}
\begin{proof}
%Let $G(X)$ be the coarse groupoid of $X$. By \cite[Proposition 3.2]{STY}, $G(X)$ is a locally compact Hausdorff $\acute{\rm e}$tale principal groupoid with unit space $\beta X$. Moreover, the closures of the graphs of controlled partial translations form a basis of compact open bisections, so $G(X)$ is ample. Since $\beta X$ is compact,
Let $G(X)$ be the coarse groupoid of $X$. By \cite[Proposition 3.2]{STY}, $G(X)$ is a locally compact Hausdorff $\acute{\rm e}$tale  principal groupoid whose unit space is canonically identified with $\beta X$. Since $X$ is discrete, $\beta X$ is compact and totally disconnected. Hence, $G(X)$ is ample. Now
$G(X)$ satisfies all the hypotheses
of \cite[Proposition A.8 (ii)]{WY}. Therefore, the complexity rank of $C_r^*(G(X))$ is bounded by the dynamical complexity rank of $G(X)$, while \cite[Example A.9]{WY}  shows the dynamical complexity rank of $G(X)$ coincides with the decomposition complexity rank of $X$. Finally, the canonical isomorphism between the reduced $C^*$-algebra of the coarse groupoid $C_r^*(G(X))$ and $C_u^*(X)$ is given in \cite[Proposition 10.29]{Roe}.
\end{proof}

%Since $\operatorname{asdim}(X)\leq 1$, for every $r>0$ there is %a decomposition $X=X_0\cup X_1$ such that each $X_i$ is a union %of an $r$-disjoint uniformly bounded family. It follows directly %from [11, Definition 2.2.1] that $X$ has %finite-decomposition-complexity rank at most one

Furthermore, $X$ having asymptotic dimension one implies that the decomposition complexity rank of $X$ is at most one \cite{GTY2}. These facts naturally lead to the following result.
\begin{theorem}
Let $X$ be a bounded geometry metric space with asymptotic dimension at most one. Then the complexity rank of
$C_u^*(X)$ is at most one, and hence, the real rank of $C_u^*(X)$ is zero. If the asymptotic dimension of $X$ is one, the complexity rank of
$C_u^*(X)$ is exactly one.
\end{theorem}
\begin{proof}
By Proposition \ref{Prop: ranks}, the complexity rank of
$C_u^*(X)$ is at most one. Theorem \ref{main thm} implies $C_u^*(X)$ has real rank zero.
  If the complexity rank of $C_u^*(X)$ is zero, then $C_u^*(X)$ would be locally finite-dimensional. By \cite[Theorem 2.2]{LW}, this would imply that $X$ has asymptotic dimension zero, a contradiction. Thus, if the asymptotic dimension of $X$ is one, the complexity rank of $C_u^*(X)$  cannot be zero.
\end{proof}

 As mentioned in \cite[Remark 3.17]{JW}, Theorem \ref{main thm} gives an affirmative answer to Question 3.10 in \cite{LW}.
 \begin{corollary}\label{Roe rr0}
 The uniform Roe algebra $C_u^*|\mathbb{Z}|$ has real rank zero.
 \end{corollary}
%\begin{proof}
% From  \cite[Ex. A.9]{WY},  if $X$ is a bounded geometry metric space, the complexity rank of the uniform Roe algebra $C_u^*|X|$ is bounded above by the geometric complexity of $X$ in the sense of \cite{GTY}.  Thus, $C_u^*|\mathbb{Z}|$ has complexity rank no more than one therefore, by Theorem \ref{main thm}, it has real rank zero.
%\end{proof}

\iffalse
\begin{example}
 The construction lies in \cite[Example 3.8]{JW}. For each $k \in \mathbb{N}$, let $X_k=\left\{m \in \mathbb{Z}\mid |m| \leq k\right\}$ equipped with the subspace metric. Let $X:=\bigsqcup_{k \in \mathbb{N}} X_k$, equipped with any metric that restricts to the given metric on $X_k$, and that satisfies $d\left(X_k, X \backslash X_k\right) \rightarrow \infty$ as $k \rightarrow \infty$. Then using the limit space machinery of \cite{SW}, one can show that $C_u^*(X)$ admits a quotient $*$-homomorphism onto $C_u^*\left|\mathbb{Z}\right|$.  As real rank zero clearly passes to quotients, Theorem \ref{Roe rr0} implies that  $C_u^*(X)$ has real rank zero,  the results  of \cite{Wei} show that $C_u^*|X|$ is stably finite.
 By \cite[Theorem 2.2]{LW},  $C_u^*|X|$  doesn't have cancellation (stable rank one). This gives an example of a stably finite, real rank zero $C^*$-algebra which doesn't have cancellation (stable rank one) and answers a long standing open question \cite[p. 455]{B}.
\end{example}
\fi

Since the free group $\mathbb{F}_{n}$ ($n=2,3,\cdots$) has asymptotic dimension one (see \cite[Theorem 2]{BDK} and \cite{Roe}), the following corollary answers the question concerning $C_u^*|\mathbb{F}_2|$  raised in \cite[p. 111]{LW}.
\begin{corollary}
  For every $n=2,3,\cdots$, $C_u^*|\mathbb{F}_n|$ has real rank zero.
\end{corollary}
\begin{remark}
  For the countably generated free group $\mathbb{F}_\infty$, if it is equipped with any proper left-invariant metric, then it has bounded geometry and asymptotic dimension one \cite[Theorem 2.1]{DS}. In this setting, $C_u^*|\mathbb{F}_{\infty}|$ has real rank zero.
\end{remark}

\begin{remark}
The uniform Roe algebra $C_u^*|\mathbb{Z}^2|$ has complexity rank
at most two by Proposition \ref{Prop: ranks},  and does not have real rank zero by \cite[Theorem 3.1]{LW}. Therefore, by Theorem \ref{main thm}, $C_u^*|\mathbb{Z}^2|$ must have complexity rank two. In fact, the stable rank of $C_u^*|\mathbb{Z}^n|$ equals two for all $n$, see \cite[Remark 2.5]{LW}.
\end{remark}

%\section*{Acknowledgements}

{\bf AI statement}
The authors gratefully acknowledge ChatGPT 5.6 Sol for its valuable suggestions for the Bochner integral and the Sylvester estimate. It is also used for language editing. All arguments and ideas are developed and substantially revised by the authors with the exception of preliminary drafts of the proofs of Lemma 2.7 which were developed with the help of ChatGPT. The authors take full responsibility of the content of this paper.

%The research of first author was supported by NNSF of China (No.:12101113,
%No.:11920101001) and the Fundamental Research Funds for the Central Universities (No.:2412021QD001). The second author was supported by NNSF of China (No.:12101102).

\end{document}